\pdfoutput=1

\documentclass{colt2016} 

\title[First-order Methods for Geodesically Convex Optimization]{First-order Methods for Geodesically Convex Optimization}
\usepackage{times}
\newcommand{\Exp}{\mathrm{Exp}}
\usepackage{amssymb}
\usepackage{amsmath}
\newcommand{\Mc}{\mathcal{M}}

\newcommand{\reals}{\mathbb{R}}

 \coltauthor{\Name{Hongyi Zhang} \Email{hongyiz@mit.edu}\\
 \addr Massachusetts Institute of Technology, Cambridge, MA 02139
 \AND
 \Name{Suvrit Sra} \Email{suvrit@mit.edu}\\
 \addr Massachusetts Institute of Technology, Cambridge, MA 02139
 }

\begin{document}

\maketitle

\begin{abstract}
	Geodesic convexity generalizes the notion of (vector space) convexity to nonlinear metric spaces. But unlike convex optimization, geodesically convex (g-convex) optimization is much less developed. In this paper we contribute to the understanding of g-convex optimization by developing iteration complexity analysis for several first-order algorithms on Hadamard manifolds. Specifically, we prove upper bounds for the global complexity of deterministic and stochastic (sub)gradient methods for optimizing smooth and nonsmooth g-convex functions, both with and without strong g-convexity. Our analysis also reveals how the manifold geometry, especially \emph{sectional curvature}, impacts convergence rates. To the best of our knowledge, our work is the first to provide global complexity analysis for first-order algorithms for general g-convex optimization. 


\end{abstract}


\section{Introduction}
Convex optimization is fundamental to numerous areas including machine learning. Convexity often helps guarantee polynomial runtimes and enables robust, more stable numerical methods. But almost invariably, the use of convexity in machine learning is limited to vector spaces, even though convexity \emph{per se} is not limited to vector spaces. Most notably, it generalizes to \emph{geodesically convex} metric spaces~\citep{gromov1978manifolds,bridson1999metric,burago2001course}, through which it offers a much richer setting for developing mathematical models amenable to global optimization.



Our broader aim is to increase awareness about g-convexity (see Definition \ref{def:g-cvx}); while our specific focus in this paper is on contributing to the understanding of geodesically convex (g-convex) optimization. In particular, we study first-order algorithms for smooth and nonsmooth g-convex optimization, for which we prove iteration complexity upper bounds. Except for a fundamental lemma that applies to general g-convex metric spaces, we limit our discussion to Hadamard manifolds (Riemannian manifolds with global nonpositive curvature), as they offer the most convenient grounds for generalization while also being relevant to numerous applications (see e.g.,~Section~\ref{sec:motiv}).


Specifically, we study optimization problems of the form
\begin{equation}
  \label{eq:main}
  \min\quad f(x)\qquad \text{such that}\quad x \in \Mc,
\end{equation}
where $f:\Mc \to \reals\cup\{\infty\}$ is a proper g-convex function and $\Mc$ is a Hadamard manifold \citep{bishop1969manifolds,gromov1978manifolds}. We solve~\eqref{eq:main} via first-order methods under a variety of settings analogous to the Euclidean case: nonsmooth, Lipschitz-smooth, and strongly g-convex. We present results for both deterministic and stochastic (where $f(x) = \mathbb{E}[F(x,\xi)]$) g-convex optimization. 

Although Riemannian geometry provides tools that enable generalization of Euclidean algorithms~\citep{udriste1994convex,absil2009optimization}, to obtain iteration complexity bounds we must overcome some fundamental geometric hurdles. We introduce key results that overcome some of these hurdles, and pave the way to analyzing first-order g-convex optimization algorithms.


\subsection{Related work and motivating examples}
\label{sec:motiv}
We recollect below a few items of related work and some examples relevant to machine learning,  where g-convexity and more generally Riemannian optimization play an important role.

Standard references on Riemannian optimization are \citep{udriste1994convex,absil2009optimization}, who primarily consider   problems on manifolds without necessarily having access to g-convexity. Consequently, their analysis is limited to asymptotic convergence (except for ~\cite[Theorem 4.2,][]{udriste1994convex} that proves linear convergence for  functions with positive-definite and bounded Riemannian Hessians). The recent monograph \citep{bacak2014convex} is devoted to g-convexity and g-convex optimization on geodesic metric spaces, though without any attention to global complexity analysis. \citet{bacak2014convex} also details a noteworthy application: averaging trees in the geodesic metric space of phylogenetic trees~\citep{billera2001geometry}.

At a more familiar level, implicitly the topic of ``geometric programming''~\citep{boyd2007tutorial} may be viewed as a special case of g-convex optimization~\citep{sra2015conic}. For instance, computing stationary states of Markov chains (e.g., while computing PageRank) may be viewed as  g-convex optimization problems by placing suitable geometry on the positive orthant; this idea has a fascinating extension to nonlinear iterations on convex cones (in Banach spaces) endowed with the structure of a geodesic metric space~\citep{lemmens2012nonlinear}. 

Perhaps the most important example of such metric spaces is the set of positive definite matrices viewed as a Riemannian or Finsler manifold; a careful study of this setup was undertaken by~\citet{sra2015conic}. They also highlighted applications to maximum likelihood estimation for certain non-Gaussian (heavy- or light-tailed) distributions, resulting in various g-convex and nonconvex likelihood problems; see also \citep{wiesel2012geodesic,zhang2013multivariate}. However, none of these three works presents a global convergence rate analysis for their algorithms.

There exist several nonconvex problems where Riemannian optimization has proved quite useful, e.g., low-rank matrix and tensor factorization~\citep{vandereycken2013low,ishteva2011best,mishra2013low}; dictionary learning~\citep{sun2015complete,harandi2012sparse}; optimization under orthogonality constraints \citep{edelman1998geometry,moakher2002means,shen2009fast,liu2015maximization};  and Gaussian mixture models~\citep{hoSr15b}, for which g-convexity helps  accelerate manifold optimization to greatly outperform the Expectation Maximization (EM) algorithm.

\subsection{Contributions}
We summarize the main contributions of this paper below.

\begin{list}{--}{\leftmargin=1em}
\setlength{\itemsep}{0pt}
\item We develop a new inequality (Lemma \ref{th:distance_bound}) useful for analyzing the behavior of optimization algorithms for functions in Alexandrov space with curvature bounded below, which can be applied to (not necessarily g-convex) optimization problems on Riemannian manifolds and beyond.
\item For g-convex optimization problems on Hadamard manifold (Riemannian manifold with nonpositive sectional curvature), we prove iteration complexity upper bounds for several existing algorithms (Table \ref{tb:results}). For the special case of smooth geodesically strongly convex optimization, a prior linear convergence result that uses line-search is known~\citep{udriste1994convex}; our results do not require line search. Moreover, as far as we are aware, ours are the first global complexity results for general g-convex optimization.
\end{list}





{ 
	\begin{table}[th]\small 
		\begin{center}
			\begin{tabular}{ccccccc}
				$f$ & Algorithm & Stepsize & Rate\footnotemark[2]& Averaging\footnotemark[3] & Theorem \vspace{4pt} \\
				\hline \hline
				\parbox[c]{1.8cm}{g-convex, \\Lipschitz} & \parbox[c]{1.8cm}{subgradient} & \parbox[c]{2cm}{$$\frac{D}{L_f\sqrt{ct}}$$} & \parbox[c]{4cm}{$$O\left(\sqrt{\frac{c}{t}}\right)$$} & \parbox[c]{0.5cm}{Yes} &  \parbox[c]{0.2cm}{\ref{th:cvx_rate}} \\
				\hline
				\parbox[c]{1.8cm}{g-convex, \\bounded \\subgradient} & \parbox[c]{1.8cm}{stochastic \\subgradient} & \parbox[c]{2cm}{$$\frac{D}{G\sqrt{ct}}$$} & \parbox[c]{4cm}{$$O\left(\sqrt{\frac{c}{t}}\right)$$} & \parbox[c]{0.5cm}{Yes} &  \parbox[c]{0.2cm}{\ref{th:stocvx_rate}} \\
				\hline
				\parbox[c]{1.8cm}{g-strongly \\convex, \\Lipschitz} & \parbox[c]{1.8cm}{subgradient} & \parbox[c]{2cm}{$$\frac{2}{\mu(s+1)}$$} & \parbox[c]{4cm}{$$O\left(\frac{c}{t}\right)$$} & \parbox[c]{0.5cm}{Yes} &  \parbox[c]{0.2cm}{\ref{th:stronglycvx_rate}} \\
				\hline
				\parbox[c]{1.8cm}{g-strongly \\convex, \\bounded \\subgradient} & \parbox[c]{1.8cm}{stochastic \\subgradient} & \parbox[c]{2cm}{$$\frac{2}{\mu(s+1)}$$} & \parbox[c]{4cm}{$$O\left(\frac{c}{t}\right)$$} & \parbox[c]{0.5cm}{Yes} &  \parbox[c]{0.2cm}{\ref{th:stostronglycvx_rate}} \\
				\hline
				\parbox[c]{1.8cm}{g-convex, \\smooth} & \parbox[c]{1.8cm}{gradient} & \parbox[c]{2cm}{$$\frac{1}{L_g}$$} & \parbox[c]{4cm}{$$O\left(\frac{c}{c+t}\right)$$} & \parbox[c]{0.5cm}{No} &  \parbox[c]{0.2cm}{\ref{th:smoothcvx_rate}} \\
				\hline
				\parbox[c]{1.8cm}{g-convex, \\smooth \\bounded \\variance} &\parbox[c]{1.8cm}{stochastic \\gradient} & \parbox[c]{2cm}{$$\frac{1}{L_g+\frac{\sigma}{D}\sqrt{ct}}$$} & \parbox[c]{4cm}{$$O\left(\frac{c+\sqrt{ct}}{c+t}\right)$$} & \parbox[c]{0.5cm}{Yes} &  \parbox[c]{0.2cm}{\ref{th:stosmoothcvx_rate}} \\
				\hline
				\parbox[c]{1.8cm}{g-strongly \\convex, \\smooth} & \parbox[c]{1.8cm}{gradient} & \parbox[c]{2cm}{$$\frac{1}{L_g}$$} & \parbox[c]{4cm}{$$O\left(\left(1-\min\left\{\frac{1}{c},\frac{\mu}{L_g}\right\}\right)^{t}\right)$$} & \parbox[c]{0.5cm}{No} &  \parbox[c]{0.2cm}{\ref{th:smoothstronglycvx_rate}}
			\end{tabular}
		\end{center}
		\caption{{\bf Summary of results.} This table summarizes the non-asymptotic convergence rates we have proved for various geodesically convex optimization algorithms. $s$: iterate index; $t$: total number of iterates; $D$: diameter of domain; $L_f$: Lipschitz constant of $f$; $c$: a constant dependent on $D$ and on the sectional curvature lower bound $\kappa$; $G$: upper bound of gradient norms; $\mu$: strong convexity constant of $f$; $L_g$: Lipschitz constant of the gradient; $\sigma$: square root variance of the gradient. } \label{tb:results}
	\end{table}
}
\footnotetext[2]{Here for simplicity only the dependencies on $c$ and $t$ are shown, while other factors are considered constant and thus omitted. Please refer to the theorems for complete results.}
\footnotetext[3]{``Yes'': result holds for proper averaging of the iterates; ``No'': result holds for the last iterate. Please refer to the theorems for complete results.}

\section{Background}
Before we describe the algorithms and analyze their properties, we would like to introduce some concepts in metric geometry and Riemannian geometry that generalize  concepts in Euclidean space. 
\subsection{Metric Geometry}
    For generalization of nonlinear optimization methods to metric space, we now recall some basic concepts in metric geometry, which cover vector spaces and Riemannian manifolds as special cases. A \emph{metric space} is a pair $(X,d)$ of set $X$ and distance function $d$ that satisfies positivity, symmetry, and the triangle inequality~\citep{burago2001course}. A continuous mapping from the interval $[0,1]$ to $X$ is called a \emph{path}. The \emph{length} of a path $\gamma : [0,1] \to X$ is defined as
    $ \mathrm{length(\gamma)} := \mathrm{sup} \sum_{i=1}^n d(\gamma(t_{i-1}), \gamma(t_i)),$
    where the supremum is taken over the set of all partitions $0=t_0<\cdots<t_n=1$ of the interval $[0,1]$, with an arbitrary $n\in \mathbb{N}$. A metric space is a \emph{length space} if for any $x,y\in X$ and $\epsilon >0$ there exists a path $\gamma: [0,1]\to X$ joining $x$ and $y$ such that $\mathrm{length}(\gamma) \leq d(x,y) + \epsilon$. A path $\gamma :[0,1]\to X$ is called a \emph{geodesic} if it is parametrized by the arc length. If every two points $x,y\in X$ are connected by a geodesic, we say $(X,d)$ is a \emph{geodesic space}. If the geodesic connecting every $x,y\in X$ is unique, the space is called \emph{uniquely geodesic}~\citep{bacak2014convex}.
    
    The properties of \emph{geodesic triangles} will be central to our analysis of optimization algorithms. A geodesic triangle $\triangle pqr$ with vertices $p,q,r\in X$ consists of three geodesics $\overline{pq} , \overline{qr}, \overline{rp}$. Given $\triangle pqr \in X$, a \emph{comparison triangle} $\triangle \bar p \bar q \bar r$ in $k$-plane is a corresponding triangle with the same side lengths in two-dimensional space of constant Gaussian curvature $k$. A length space with curvature bound is called an Alexandrov space. In particular, we have the following important definition:
    \begin{definition}[Alexandrov space with curvature $\ge k$] \label{def:alexandrov}
    	Let $k$ be a real number. A length space $X$ is a space of curvature $\ge k$ if every point $x\in X$ has a neighborhood $U$ such that for any triangle $\triangle abc$ contained in $U$ and any point $d\in\overline{ac}$ the inequality $|bd|\ge|\bar b\bar d|$ holds, where $\triangle \bar a\bar b\bar c$ is a comparison triangle in the $k$-plane and $\bar d\in\overline{\bar a\bar c}$ is the point such that $|\bar a\bar d| = |ad|$.
	\end{definition}
    The notion of angle is defined in the following sense. Let $\gamma: [0,1]\to X$ and $\eta: [0,1]\to X$ be two geodesics in $(X,d)$ with $\gamma_0 = \eta_0$, we define the \emph{angle} between $\gamma$ and $\eta$ as $ \alpha(\gamma, \eta) := \lim\sup_{s,t\to 0_+} \measuredangle \bar \gamma_s \bar \gamma_0 \bar \eta_t $
    where $\measuredangle \bar \gamma_s \bar \gamma_0 \bar \eta_t$ is the angle at $\bar \gamma_0$ of the corresponding triangle $\triangle \bar \gamma_s \bar \gamma_0 \bar \eta_t$. We use Toponogov's theorem to relate the angles and lengths of any geodesic triangle in a geodesic space to those of a comparison triangle in a space of constant curvature \citep{burago1992ad,burago2001course}. 
    
\subsection{Riemannian Geometry}

\begin{figure}[hbt]
	\centering \def\svgwidth{150pt}
	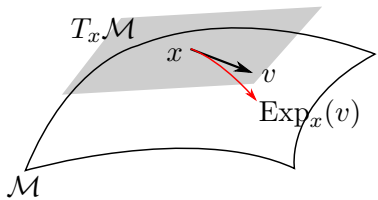
	\caption{Illustration of a manifold. Also shown are tangent space, geodesic and exponential map.}
\end{figure}

   An $n$-dimensional \emph{manifold} is a topological space where each point has a neighborhood that is homeomorphic to the $n$-dimensional Euclidean space. At any point $x$ on a manifold, tangent vectors are defined as the tangents of parametrized curves passing through $x$. The \emph{tangent space} $T_x\mathcal{M}$ of a manifold $\mathcal{M}$ at $x$ is defined as the set of all tangent vectors at the point $x$. An exponential map at $x \in \mathcal{M}$ is a mapping from the tangent space $T_x\mathcal{M}$ to $\mathcal{M}$ with the requirement that a vector $v \in T_x\mathcal{M}$ is mapped to the point $y:=\Exp_x(v)\in \mathcal{M}$ such that there exists a geodesic $\gamma : [0,1]\to \mathcal{M}$ satisfying $\gamma(0) = x, \gamma(1) = y$ and $\gamma'(0) = v$. 
   
   As tangent vectors at two different points $x, y\in \mathcal{M}$ lie in different tangent spaces, we cannot compare them directly. To meaningfully compare vectors in different tangent spaces, one needs to define a way to move a tangent vector along the geodesics, while `preserving' its length and orientation. We thus need to use an inner product structure on tangent spaces, which is called a Riemannian metric. A \emph{Riemannian manifold} $(\mathcal{M}, g)$ is a real smooth manifold equipped with an inner product $g_x$ on the tangent space $T_x\mathcal{M}$ of every point $x$, such that if $u,v$ are two vector fields on $\mathcal{M}$ then $x \mapsto \langle u, v \rangle_x := g_x(u,v)$ is a smooth function. On a Riemannian manifold, the notion of \emph{parallel transport} (parallel displacement) provides a sensible way to transport a vector along a geodesic. Intuitively, a tangent vector $v\in T_x\mathcal{M}$ at $x$ of a geodesic $\gamma$ is still a tangent vector $\Gamma(\gamma)_x^y v$ of $\gamma$ after being transported to a point $y$ along $\gamma$. Furthermore, parallel transport preserves inner products, i.e. $\langle u, v \rangle_x = \langle \Gamma(\gamma)_x^y u, \Gamma(\gamma)_x^y v \rangle_y$. 
   
    The curvature of a Riemannian manifold is characterized by its Riemannian metric tensor at each point. For worst-case analysis, it is sufficient to consider geodesic triangles of any two-dimensional subspace.  \emph{Sectional curvature} is the Gauss curvature of a two dimensional subspace of a Riemannian manifold, which characterizes the metric space property within that subspace. A subspace with positive, zero or negative sectional curvature is locally isometric to a two dimensional sphere, a Euclidean plane, or a hyperbolic plane with the same Gauss curvature.

\subsection{Function Classes on a Riemannian Manifold}
	We first define some key terms. Throughout the paper, we assume that  the function $f$ is defined on a Riemannian manifold $\mathcal{M}$, unless stated otherwise.
	
	\begin{definition}[Geodesic convexity] \label{def:g-cvx}
		A function $f:\mathcal{M}\to\mathbb{R}$ is said to be geodesically convex if for any $x,y\in\mathcal{M}$, a geodesic $\gamma$ such that $\gamma(0)=x$ and $\gamma(1)=y$, and $t\in [0,1]$, it holds that
		\[ f(\gamma(t)) \le (1-t)f(x) + tf(y). \]
	\end{definition}
	It can be shown that an equivalent definition is that for any $x,y\in\mathcal{M}$,
	\[ f(y) \ge f(x) + \langle g_x, \Exp_x^{-1}(y) \rangle_x, \]
	where $g_x$ is a subgradient of $f$ at $x$, or the gradient if $f$ is differentiable, and $\langle\cdot,\cdot\rangle_x$ denotes the inner product in the tangent space of $x$ induced by the Riemannian metric. In the rest of the paper we will omit the index of tangent space when it is clear from the context.
	\begin{definition}[Strong convexity]
		A function $f:\mathcal{M}\to\mathbb{R}$ is said to be geodesically $\mu$-strongly convex if for any $x,y\in\mathcal{M}$,
		\[ f(y) \ge f(x) + \langle g_x, \Exp_x^{-1}(y) \rangle_x + \frac{\mu}{2}d^2(x,y).\]
	\end{definition}
	\begin{definition}[Lipschitzness] 
		A function $f:\mathcal{M}\to\mathbb{R}$ is said to be geodesically $L_f$-Lipschitz if for any $x,y\in\mathcal{M}$,
		\[ |f(x)-f(y)| \le L_f d(x,y). \]
	\end{definition}
	\begin{definition}[Smoothness]
		A differentiable function $f:\mathcal{M}\to\mathbb{R}$ is said to be geodesically $L_g$-smooth if its gradient is $L_g$-Lipschitz, i.e. for any $x,y\in\mathcal{M}$,
		\[ \|g_x - \Gamma_y^x g_y\| \le L_g d(x,y) \]
	where $\Gamma_y^x$ is the parallel transport from $y$ to $x$.
	\end{definition}
        Observe that compared to the Euclidean setup, the above definition requires a parallel transport operation to ``transport'' $g_y$ to $g_x$. It can be proved that if $f$ is $L_g$-smooth, then for any $x,y\in\mathcal{M}$,
	\[ f(y) \le f(x) + \langle g_x, \Exp_x^{-1}(y) \rangle_x + \frac{L_g}{2}d^2(x,y). \]

\section{Convergence Rates of First-order Methods} 

General subgradient / gradient algorithms on Riemannian manifolds take the form
\begin{equation}
  x_{s+1} = \Exp_{x_s}(-\eta_s g_s),
\end{equation}
where $s$ is the iterate index, $g_s$ is a subgradient of the objective function, and $\eta_s$ is a step-size. For brevity, we will use the word `gradient' to refer to both subgradient and gradient, deterministic or stochastic; the meaning should be apparent from the context. 

While it is easy to translate first-order optimization algorithms from Euclidean space to Riemannian manifolds, and similarly to prove asymptotic convergence rates (since locally Riemannian manifolds resemble Euclidean space), it is much harder to carry out \emph{non-asymptotic} analysis, at least due to the following two difficulties:
 \begin{itemize}
	 \item {\bf Non-Euclidean trigonometry is difficult to use.} Trigonometric geometry in nonlinear spaces is fundamentally different from Euclidean space. In particular, for analyzing optimization algorithms, the law of cosines in Euclidean space 
		 \begin{equation} \label{eq:euclidean_loc}
			 a^2 = b^2 + c^2 - 2bc\cos(A),
		 \end{equation} 
	 where $a,b,c$ are the sides of a Euclidean triangle with $A$ the angle between sides $b$ and $c$, is an essential tool for bounding the squared distance between the iterates and the minimizer(s). Indeed, consider the Euclidean update $x_{s+1} = x_s - \eta_s g_s$. Applying (\ref{eq:euclidean_loc}) to the triangle $\triangle x_s x x_{s+1}$, with $a=\overline{xx_{s+1}}$, $b=\overline{x_sx_{s+1}}$, $c=\overline{xx_s}$, and $A=\measuredangle x x_s x_{s+1}$, we get the frequently used formula
		 \[ \|x_{s+1}-x\|^2 = \|x_s-x\|^2 - 2\eta_s\langle g_s, x_s-x \rangle + \eta_s^2\|g_s\|^2 \]
	 However, this nice equality does not exist for nonlinear spaces.
	 \item {\bf Linearization does not work.} Another key technique used in bounding squared distances is inspired by the proximal algorithms. Here, gradient-like updates are seen as proximal steps for minimizing a series of \emph{linearizations} of the objective function. Specifically, let $\psi(x;x_s) = f(x_s) + \langle g_s, x-x_s \rangle$ be the linearization of the convex function $f$, and let $g_s \in \partial f(x_s)$. Then, $x_{s+1} = x_s - \eta_s g_s$ is the unique solution to the following minimization problem
		 \[ \min_x\ \Bigl\{\psi(x;x_s) + \frac{1}{2\eta_s}\|x-x_s\|^2 \Bigr\}. \]
	 Since $\psi(x;x_s)$ is convex, we thus have (see e.g. \citet{tseng2009accelerated}) the recursively useful bound
		 \[ \psi(x_{s+1};x_s) + \frac{1}{2\eta_s}\|x_{s+1}-x\|^2 \le \psi(x;x_s) + \frac{1}{2\eta_s}\|x_s-x\|^2 - \frac{\eta_s}{2}\|g_s\|^2. \]
	 But in nonlinear space there is no trivial analogy of a linear function. For example, for any given $y\in\mathcal{M}$ and $g_y\in T_y\mathcal{M}$, the function 
	 \[ \psi(x;y) = f(y) + \langle g_y, \Exp_y^{-1}(x) \rangle, \] 
	 is geodesically both star-concave and star-convex in $y$, but neither convex nor concave in general. Thus a nonlinear analogue of the above result does not hold.
\end{itemize}

We address the first difficulty by developing an easy-to-use trigonometric distance bound for Alexandrov space with curvature bounded below. When specialized to Hadamard manifolds, our result reduces to the analysis in \citep{bonnabel2013stochastic}, which in turn relies on~ \citep[][Lemma 3.12]{cordero2001riemannian}. However, unlike \citep{cordero2001riemannian}, our proof assumes no manifold structure on the geodesic space of interest, and is fundamentally different in techniques.

 \subsection{Trigonometric Distance Bound}
As noted above, a main hurdle in analyzing non-asymptotic convergence of first-order methods in geodesic spaces is that the Euclidean law of cosines does not hold any more. For general nonlinear spaces, there are no corresponding analytical expressions. Even for the (hyperbolic) space of constant negative curvature $-1$, perhaps the simplest and most studied nonlinear space, the law of cosines is replaced by the \emph{hyperbolic law of cosines}:
\begin{equation}
  \label{eq:7}
  \cosh a = \cosh b \cosh c - \sinh b \sinh c \cos(A),
\end{equation}
which is not amendable to the standard techniques of convergence rate analysis. With the goal of developing analysis for nonlinear space optimization algorithms, our first contribution is the following trigonometric distance bound for Alexandrov space with curvature bounded below. Owing to its fundamental nature, we believe that this lemma may be of broader interest too.
\begin{lemma} \label{th:distance_bound}
	If $a,b,c$ are the sides (i.e., side lengths) of a geodesic triangle in an Alexandrov space with curvature lower bounded by $\kappa$, and $A$ is the angle between sides $b$ and $c$, then
	\begin{equation} \label{eq:distance_bound}
		a^2 \leq \frac{\sqrt{|\kappa|}c}{\tanh(\sqrt{|\kappa|}c)}b^2 + c^2 - 2bc\cos(A).
	\end{equation}
\end{lemma}
\begin{proof}\textbf{sketch.} The complete proof contains technical details that digress from the main focus of this paper, so we leave them in the appendix. Below we sketch the main steps.
	
	Our first observation is that by the famous Toponogov's theorem \citep{burago1992ad,burago2001course}, we can upper bound the side lengths of a geodesic triangle in an Alexandrov space with curvature bounded below by the side lengths of a comparison triangle in the hyperbolic plane, which satisfies (\emph{cf.}~\eqref{eq:7}):
		\begin{equation} \label{eq:hyperbolic_law_of_cosines}
			\cosh (\sqrt{|\kappa|}a) = \cosh (\sqrt{|\kappa|}b) \cosh (\sqrt{|\kappa|}c) - \sinh (\sqrt{|\kappa|}b) \sinh (\sqrt{|\kappa|}c) \cos(A).
		\end{equation} 
	Second, we observe that it suffices to study $\kappa = -1$, which corresponds to~\eqref{eq:7}, 
	since Eqn.~(\ref{eq:hyperbolic_law_of_cosines}) can be seen as  Eqn.~\eqref{eq:7} 
        with side lengths $a = \sqrt{|\kappa|}a', b = \sqrt{|\kappa|}b', c = \sqrt{|\kappa|}c'$ (see Lemma \ref{th:hyperbolic_distance_bound}).
	
	Finally, we observe that in~\eqref{eq:7},  
        $\frac{\partial^2}{\partial b^2} \cosh(a) = \cosh(a)$. Letting $g(b,c,A) := \cosh(\sqrt{\text{rhs}(b,c,A)})$, where $\text{rhs}(b,c,A)$ is the right hand side of \eqref{eq:distance_bound}, we then see that it is sufficient to prove the following:
	\begin{enumerate}
          \setlength{\itemsep}{0pt}
		\item $\cosh(a)$ and $g(b,c,A)$ are equal at $b=0$.
		\item the first partial derivatives of $\cosh(a)$ and $g(b,c,A)$ w.r.t.~$b$ agree at $b=0$.
		\item  $\frac{\partial^2}{\partial b^2} g(b,c,A) \ge g(b,c,A)$ for $b,c\ge 0$ (Lemma \ref{th:lemma_super_exponential}).
	\end{enumerate}
        These three steps, if true, lead to the proof of $\cosh(a)\le g(b,c,A)$ for $b,c\ge 0$, thus proving a special case of Lemma \ref{th:distance_bound} for space with constant sectional curvature $-1$ as shown in Lemma \ref{th:lemma_differential_inequality}, \ref{th:canonic_hyperbolic_distance_bound}. Combing this special case with our first two observations concludes the proof of the lemma. 
\end{proof}

\begin{remark}
  Inequality~(\ref{eq:distance_bound}) provides an upper bound on the side lengths of a geodesic triangle in an Alexandrov space with curvature bounded below. Some examples of such spaces are Riemannian manifolds, including hyperbolic space, Euclidean space, sphere, orthogonal groups, and compact sets on a PSD manifold. However, our derivation does not rely on any manifold structure, thus it also applies to certain cones and convex hypersurfaces \citep{burago2001course}.
\end{remark}

In the sequel, we use the notation $\zeta(\kappa,c)\triangleq \frac{\sqrt{|\kappa|}c}{\tanh(\sqrt{|\kappa|}c)}$ for the curvature dependent quantity from inequality~\eqref{eq:distance_bound}. From Lemma~\ref{th:distance_bound} it is straightforward to prove the following corollary, which characterizes an important relation between two consecutive updates of an iterative optimization algorithm on Riemannian manfiold with curvature bounded below.
\begin{corollary} \label{th:distance_corollary}
	For any Riemannian manifold $\mathcal{M}$ where the sectional curvature is lower bounded by $\kappa$ and any point $x$, $x_s\in\mathcal{M}$, the update $x_{s+1} = \Exp_{x_s}(-\eta_s g_s)$ satisfies
	\begin{equation} \label{eq:distance_corollary}
		\langle -g_s, \Exp_{x_s}^{-1}(x)\rangle \le \frac{1}{2\eta_s} \left(d^2(x_{s},x) - d^2(x_{s+1},x)\right) + \frac{\zeta(\kappa,d(x_s,x))\eta_s}{2} \|g_s\|^2.
	\end{equation}
\end{corollary}
\begin{proof}
  Simply notice that for the geodesic triangle $\triangle x_s x_{s+1} x$, we have $d(x_s,x_{s+1}) = \eta_s\|g_s\|$, while  $d(x_s,x_{s+1})d(x_s,x)\cos(\measuredangle x_{s+1}x_s x) = \langle -\eta_sg_s,\Exp_{x_s}^{-1}(x) \rangle$. Now $a = \overline{x_{s+1}x}, b = \overline{x_{s+1}x_s}, c = \overline{x_sx}, A = \measuredangle x_{s+1}x_s x$, apply Lemma \ref{th:distance_bound} and simplify to obtain~\eqref{eq:distance_corollary}.
\end{proof}

It is instructive to compare \eqref{eq:distance_corollary} with its Euclidean counterpart (for which actually $\zeta=1$):
	\[ \langle -g_s, x-x_s \rangle = \tfrac{1}{2\eta_s} \left(\|x_{s}-x\|^2 - \|x_{s+1}-x\|^2\right) + \tfrac{\eta_s}{2} \|g_s\|^2. \]
Corollary \ref{th:distance_corollary} furnishes the missing tool for analyzing non-asymptotic convergence rates of manifold optimization algorithms. We now move to the analysis of several such first-order algorithms.
\subsection{Convergence Rate Analysis}
\paragraph{Nonsmooth convex optimization.} The following two theorems show that both deterministic and stochastic subgradient methods achieve a \emph{curvature-dependent} $O(1/\sqrt{t})$ rate of convergence for g-convex on Hadamard manifolds.
\begin{theorem} \label{th:cvx_rate}
  Let $f$ be g-convex and $L_f$-Lipschitz, the diameter of domain be bounded by $D$, and the sectional curvature lower-bounded by $\kappa\le 0$. Then, the subgradient method with a constant stepsize $\eta_s = \eta = \frac{D}{L_f\sqrt{\zeta(\kappa,D)t}}$ and $\overline{x}_1 = x_1$,$\overline{x}_{s+1} = \Exp_{\overline{x}_s}\left(\frac{1}{s+1}\Exp_{\overline{x}_s}^{-1}(x_{s+1})\right)$  satisfies
	\[ f\left(\overline{x}_{t} \right) - f(x^*) \leq DL_f\sqrt{\frac{\zeta(\kappa,D)}{t}}. \]
\end{theorem}
\begin{proof}
	Since $f$ is g-convex, it satisfies $f(x_{s}) - f(x^*) \le \langle-g_s, \Exp_{x_s}^{-1}(x^*),\rangle$ which combined with Corollary \ref{th:distance_corollary} and the $L_f$-Lipschitz condition yields the upper bound
	\begin{equation} \label{eq:subgrad_telescoping_element}
		f(x_{s}) - f(x^*) \le \frac{1}{2\eta} \left(d^2(x_{s},x^*) - d^2(x_{s+1},x^*)\right) + \frac{\zeta(\kappa,D)L_f^2\eta}{2}.
	\end{equation}
	Summing over $s$ from $1$ to $t$ and dividing by $t$, we obtain
		\begin{equation} \label{eq:subgrad_telescoping_sum}
			\frac{1}{t}\sum_{s=1}^t f(x_{s}) - f(x^*) \leq \frac{1}{2t\eta}\left( d^2(x_1,x^*) - d^2(x_{t+1},x^*) \right) + \frac{\zeta(\kappa,D)L_f^2\eta}{2}.
		\end{equation}
       Plugging in $d(x_1,x^*)\leq D$ and $\eta = \frac{D}{L_f\sqrt{\zeta(\kappa,D)t}}$ we further obtain
		\[ \frac{1}{t}\sum_{s=1}^t f(x_{s}) - f(x^*) \leq DL_f\sqrt{\frac{\zeta(\kappa,D)}{t}}. \]
	It remains to show that $f(\overline{x}_t) \le \frac{1}{t}\sum_{s=1}^t f(x_{s})$, which can be proved by an easy induction.
\end{proof}

We note that Theorem~\ref{th:cvx_rate} and our following results are all generalizations of known results in Euclidean space. Indeed, setting curvature $\kappa=0$ we can recover the Euclidean convergence rates (in some cases up to a difference in small constant factors). However, for Hadamard manifolds $\kappa<0$ and the theorem implies that the algorithms may converge more slowly. Also worth noting is that we must be careful in how we obtain the ``average'' iterate $\overline{x}_t$ on the manifold.

\begin{theorem} \label{th:stocvx_rate}
	If $f$ is g-convex, the diameter of the domain is bounded by $D$, the sectional curvature of the manifold is lower bounded by $\kappa\le 0$, and the stochastic subgradient oracle satisfies $\mathbb{E}[\widetilde{g}(x)] = g(x) \in\partial f(x), \mathbb{E}[\|\widetilde{g}_s\|^2] \le G^2$, then the stochastic subgradient method with stepsize $\eta_s = \eta = \frac{D}{G\sqrt{\zeta(\kappa,D)t}}$, and $\overline{x}_1 = x_1, \overline{x}_{s+1} = \Exp_{\overline{x}_s}\left(\frac{1}{s+1}\Exp_{\overline{x}_s}^{-1}(x_{s+1})\right)$  satisfies the upper bound
	\[ \mathbb{E}[f\left(\overline{x}_{t} \right) - f(x^*)] \leq DG\sqrt{\frac{\zeta(\kappa,D)}{t}}. \]
\end{theorem}
\begin{proof}
	The proof structure is very similar, except that for each equation we take expectation with respect to the sequence $\{x_s\}_{s=1}^t$. Since $f$ is g-convex, we have
	\begin{eqnarray*}
		\mathbb{E}[f(x_{s}) - f(x^*)] \le \langle-\mathbb{E}[\widetilde{g}_s], \Exp_{x_s}^{-1}(x^*)\rangle,
	\end{eqnarray*}
	which combined with Corollary \ref{th:distance_corollary} and $\mathbb{E}[\|\widetilde{g}_s\|^2] \le G^2$ yields
	\begin{equation} \label{eq:stosubgrad_telescoping_element}
		\mathbb{E}[f(x_{s}) - f(x^*)] \le \frac{1}{2\eta} \mathbb{E}\left[d^2(x_{s},x^*) - d^2(x_{s+1},x^*)\right] + \frac{\zeta(\kappa,D)G^2\eta}{2}.
	\end{equation}
        Now arguing as in Theorem~\ref{th:cvx_rate} the proof follows.
\end{proof}

\paragraph{Strongly convex nonsmooth functions.} The following two theorems show that both subgradient method and stochastic subgradient method achieve a curvature dependent $O(1/t)$ rate of convergence for g-strongly convex functions on Hadamard manifolds. 
\begin{theorem} \label{th:stronglycvx_rate}
	If $f$ is geodesically $\mu$-strongly convex and $L_f$-Lipschitz, and the sectional curvature of the manifold is lower bounded by $\kappa\le 0$, then the subgradient method with $\eta_s = \frac{2}{\mu(s+1)}$  satisfies
	\[ f\left(\overline{x}_t \right) - f(x^*) \le \frac{2\zeta(\kappa,D)L_f^2}{\mu(t+1)},\]
	 where $\overline{x}_1 = x_1$, and $\overline{x}_{s+1} = \Exp_{\overline{x}_s}\left(\frac{2}{s+1}\Exp_{\overline{x}_s}^{-1}(x_{s+1})\right)$.
\end{theorem}
\begin{proof}
  Since $f$ is geodesically $\mu$-strongly convex, we have
	\begin{eqnarray*}
		f(x_{s}) - f(x^*) \le \langle-g_s, \Exp_{x_s}^{-1}(x^*)\rangle - \frac{\mu}{2}d^2(x_s,x^*),
	\end{eqnarray*}
	which combined with Corollary \ref{th:distance_corollary} and $L_f$-Lipschitz condition yields
	\begin{align}
		\label{eq:stronglycvx_remark1} f(x_{s}) - f(x^*) &\le \left(\frac{1}{2\eta_s}-\frac{\mu}{2}\right) d^2(x_{s},x^*) - \frac{1}{2\eta_s}d^2(x_{s+1},x^*) + \frac{\zeta(\kappa,D)L_f^2\eta_s}{2} \\
			\label{eq:mu_telescoping_element}& = \frac{\mu(s-1)}{4}d^2(x_{s},x^*) - \frac{\mu(s+1)}{4}d^2(x_{s+1},x^*) + \frac{\zeta(\kappa,D)L_f^2}{\mu(s+1)}.
	\end{align}
	Multiply (\ref{eq:mu_telescoping_element}) by $s$ and sum over $s$ from $1$ to $t$; then divide the result by $\frac{t(t+1)}{2}$ to obtain
	\begin{align}
		\frac{2}{t(t+1)}\sum_{s=1}^t sf(x_{s}) - f(x^*) \le \frac{2\zeta(\kappa,D)L_f^2}{\mu(t+1)}.
	\end{align}
	The final step is to show $f(\overline{x}_t) \le \frac{2}{t(t+1)}\sum_{s=1}^t sf(x_{s})$, which again follows by an easy induction.
\end{proof}

\begin{theorem} \label{th:stostronglycvx_rate}
	If $f$ is geodesically $\mu$-strongly convex, the sectional curvature of the manifold is lower bounded by $\kappa\le 0$, and the stochastic subgradient oracle satisfies $\mathbb{E}[\widetilde{g}(x)] = g(x) \in\partial f(x), \mathbb{E}[\|\widetilde{g}_s\|^2] \le G^2$, then the subgradient method with $\eta_s = \frac{2}{\mu(s+1)}$  satisfies
	\[ \mathbb{E}[f\left(\overline{x}_t \right) - f(x^*)] \le \frac{2\zeta(\kappa,D)G^2}{\mu(t+1)} \]
	where $\overline{x}_1 = x_1$, and $\overline{x}_{s+1} = \Exp_{\overline{x}_s}\left(\frac{2}{s+1}\Exp_{\overline{x}_s}^{-1}(x_{s+1})\right)$.
\end{theorem}
\begin{proof}
	The proof structure is very similar to the previous theorem, except that now we take expectations over the sequence $\{x_s\}_{s=1}^t$. We omit the details for brevity. 
\end{proof}

Theorems~\ref{th:stronglycvx_rate} and \ref{th:stostronglycvx_rate} are generalizations of their Euclidean counterparts \citep{lacoste2012simpler}, and follow the same proof structures. Our upper bounds depend linearly on $\zeta(\kappa,D)$, which implies that with $\kappa<0$ the algorithms may converge more slowly. However, note that for strongly convex problems, the distances from iterates to the minimizer are shrinking, thus the inequality (\ref{eq:stronglycvx_remark1}) (or its stochastic version)  may be too pessimistic, and better dependency on $\kappa$ may be obtained with a more refined analysis. We leave this as an open problem for the future.

\paragraph{Smooth convex optimization.} The following two theorems show that gradient descent algorithm achieves a curvature dependent $O(1/t)$ rate of convergence, whereas stochastic gradient achieves a curvature dependent $O(1/t+\sqrt{1/t})$ rate for smooth g-convex functions on Hadamard manifolds. 
\begin{theorem} \label{th:smoothcvx_rate}
	If $f:\mathcal{M}\to \mathbb{R}$ is g-convex with an $L_g$-Lipschitz gradient, the diameter of domain is bounded by $D$, and the sectional curvature of the manifold is bounded below by $\kappa$, then gradient descent with $\eta_s = \eta = \frac{1}{L_g}$ satisfies for $t>1$ the upper bound
	\[ f(x_t) - f(x^*) \leq \frac{\zeta(\kappa,D)L_g D^2}{2(\zeta(\kappa,D)+t-2)}. \]
\end{theorem}
\begin{proof}
	For simplicity we denote $\Delta_s = f(x_s) - f(x^*)$.
	First observe that with $\eta = \frac{1}{L_g}$ the algorithm is a descent method. Indeed, we have
	\begin{align} \label{eq:smooth_s+1_minus_s}
		\Delta_{s+1} - \Delta_s \leq \langle g_s, \Exp_{x_s}^{-1}(x_{s+1}) \rangle + \frac{L_g}{2} d^2(x_{s+1},x_s) = - \frac{\|g_s\|^2}{2L_g}.
	\end{align}
	On the other hand, by the convexity of $f$ and Corollary \ref{th:distance_corollary} we obtain
	\begin{align} \label{eq:smooth_s+1_minus_star}
		\Delta_s \le \langle -g_s, \Exp_{x_s}^{-1}(x^*)\rangle \le \frac{L_g}{2} \left(d^2(x_{s},x^*) - d^2(x_{s+1},x^*)\right) + \frac{\zeta(\kappa,D) \|g_s\|^2}{2L_g}.
	\end{align}
	Multiplying (\ref{eq:smooth_s+1_minus_s}) by $\zeta(\kappa,D)$ and adding to (\ref{eq:smooth_s+1_minus_star}), we get
	\begin{equation}
		\zeta(\kappa,D)\Delta_{s+1} - (\zeta(\kappa,D)-1)\Delta_s \le \frac{L_g}{2} \left(d^2(x_{s},x^*) - d^2(x_{s+1},x^*)\right).
	\end{equation}
	Now summing over $s$ from $1$ to $t-1$, a brief manipulation shows that
	\begin{equation}
		\zeta(\kappa,D)\Delta_t + \sum_{s=2}^{t-1}\Delta_s \le (\zeta(\kappa,D)-1)\Delta_1 + \tfrac{L_g D^2}{2}.
	\end{equation}
	Since for $s\le t$ we proved $\Delta_t \le \Delta_s$, and by assumption $\Delta_1 \le \frac{L_g D^2}{2}$, for $t>1$ we get
	\[ \Delta_t \le \frac{\zeta(\kappa,D)L_g D^2}{2(\zeta(\kappa,D)+t-2)}, \]
        yielding the desired bound.
\end{proof}

\begin{theorem} \label{th:stosmoothcvx_rate}
	If $f:\mathcal{M}\to \mathbb{R}$ is g-convex with $L_g$-Lipschitz gradient, the diameter of domain is bounded by $D$, the sectional curvature of the manifold is bounded below by $\kappa$, and the stochastic gradient oracle satisfies $\mathbb{E}[\widetilde{g}(x)] = g(x) = \nabla f(x), \mathbb{E}[\|\nabla f(x) - \widetilde{g}_s\|^2] \le \sigma^2$, then the stochastic gradient algorithm with $\eta_s = \eta = \frac{1}{L_g + 1/\alpha}$ where $\alpha=\frac{D}{\sigma}\sqrt{\frac{1}{\zeta(\kappa,D)t}}$ satisfies for  $t>1$
	\[ \mathbb{E}[f(\overline{x}_t) - f(x^*)] \leq \frac{\zeta(\kappa,D)L_g D^2 + 2D\sigma\sqrt{\zeta(\kappa,D)t}}{2(\zeta(\kappa,D)+t-2)}, \]
	where $\overline{x}_2 = x_2$, $\overline{x}_{s+1} = \Exp_{\overline{x}_s}\left(\frac{1}{s}\Exp_{\overline{x}_s}^{-1}(x_{s+1})\right)$ for $2\le s\le t-2$, $\overline{x}_t = \Exp_{\overline{x}_{t-1}}\left(\frac{\zeta(\kappa,D)}{\zeta(\kappa,D)+t-2}\Exp_{\overline{x}_{t-1}}^{-1}(x_t)\right)$.
\end{theorem}
\begin{proof}
  As before we write $\Delta_s = f(x_s) - f(x^*)$. First we observe that
	\begin{align} 
	\Delta_{s+1} - \Delta_s &\le \langle g_s, \Exp_{x_s}^{-1}(x_{s+1}) \rangle + \frac{L_g}{2} d^2(x_{s+1},x_s) \\
		&= \langle \widetilde{g}_s, \Exp_{x_s}^{-1}(x_{s+1}\rangle + \langle g_s - \widetilde{g}_s, \Exp_{x_s}^{-1}(x_{s+1}\rangle + \frac{L_g}{2} d^2(x_{s+1},x_s) \\
		&\le \langle \widetilde{g}_s, \Exp_{x_s}^{-1}(x_{s+1})\rangle + \frac{\alpha}{2} \|g_s - \widetilde{g}_s\|^2 + \frac{1}{2}\left(L_g+\frac{1}{\alpha}\right) d^2(x_{s+1},x_s)
	\end{align}
	Taking expectation, and letting $\eta = \frac{1}{L_g + 1/\alpha}$, we obtain
	\begin{align} \label{eq:stosmooth_s+1_minus_s}
	\mathbb{E}[\Delta_{s+1} - \Delta_s] \le  \frac{\alpha\sigma^2}{2} - \frac{\mathbb{E}[\|\widetilde{g}_s\|^2]}{2\left(L_g+\frac{1}{\alpha}\right)}.
	\end{align}
	On the other hand, using convexity of $f$ and Corollary \ref{th:distance_corollary} we get
	\begin{align} \label{eq:stosmooth_s+1_minus_star}
	\Delta_s \le \langle -g_s, \Exp_{x_s}^{-1}(x^*)\rangle \le \frac{L_g+\frac{1}{\alpha}}{2} \mathbb{E}\left[d^2(x_{s},x^*) - d^2(x_{s+1},x^*)\right] + \frac{\zeta(\kappa,D) \mathbb{E}[\|\widetilde{g}_s\|^2]}{2\left(L_g+\frac{1}{\alpha}\right)}.
	\end{align}
	Multiply (\ref{eq:stosmooth_s+1_minus_s}) by $\zeta(\kappa,D)$ and add to (\ref{eq:stosmooth_s+1_minus_star}), we get
	\begin{equation*}
		\mathbb{E}[\zeta(\kappa,D)\Delta_{s+1} - (\zeta(\kappa,D)-1)\Delta_s] \le \frac{L_g+\frac{1}{\alpha}}{2} \mathbb{E}\left[d^2(x_{s},x^*) - d^2(x_{s+1},x^*)\right] + \frac{\alpha \zeta(\kappa,D)\sigma^2}{2}.
	\end{equation*}
	Summing over $s$ from $1$ to $t-1$ and simplifying, we obtain
	\begin{equation}
		\mathbb{E}[\zeta(\kappa,D)\Delta_t + \sum_{s=2}^{t-1}\Delta_s] \le \mathbb{E}[(\zeta(\kappa,D)-1)\Delta_1] + \frac{L_g D^2}{2} + \frac{1}{2}\left(\frac{D^2}{\alpha} + \alpha \zeta(\kappa,D)t\sigma^2\right).
	\end{equation}
	Now set $\alpha = \frac{D}{\sigma\sqrt{\zeta(\kappa,D)t}}$, and note that $\Delta_1 \le \frac{L_g D^2}{2}$; thus, for $t>1$ we get
	\[ \mathbb{E}\bigl[\zeta(\kappa,D)\Delta_t + \sum\nolimits_{s=2}^{t-1}\Delta_s\bigr] \le \frac{\zeta(\kappa,D)L_g D^2}{2} + D\sigma\sqrt{\zeta(\kappa,D)t}.\]
	Finally, due to g-convexity of $f$ it is easy to verify by induction that 
		\begin{equation*}\mathbb{E}[f(\overline{x}_t)-f(x^*)] \le \frac{\mathbb{E}[\zeta(\kappa,D)\Delta_t + \sum_{s=2}^{t-1}\Delta_s]}{\zeta(\kappa,D)+t-2}.\end{equation*}%
\end{proof}

\paragraph{Smooth and strongly convex functions.} Next we prove that gradient descent achieves a curvature dependent linear rate of convergence for geodesically strongly convex and smooth functions on Hadamard manifolds.
\begin{theorem} \label{th:smoothstronglycvx_rate}
	If $f:\mathcal{M}\to \mathbb{R}$ is geodesically $\mu$-strongly convex with $L_g$-Lipschitz gradient, and the sectional curvature of the manifold is bounded below by $\kappa$, then the gradient descent algorithm with $\eta_s = \eta = \frac{1}{L_g}$, $\epsilon = \min\{\frac{1}{\zeta(\kappa,D)}, \frac{\mu}{L_g}\}$ satisfies for $t > 1$
	\[ f(x_t) - f(x^*) \leq \frac{(1-\epsilon)^{t-2}L_gD^2}{2}.\]
\end{theorem}
\begin{proof}
  As before we use $\Delta_s = f(x_s) - f(x^*)$. Observe that with $\eta = \frac{1}{L_g}$ we have descent:
	\begin{align} \label{eq:smooth+strong_s+1_minus_s}
	\Delta_{s+1} - \Delta_s \leq \langle g_s, \Exp_{x_s}^{-1}(x_{s+1}) \rangle + \frac{L_g}{2} d^2(x_{s+1},x_s) = - \frac{\|g_s\|^2}{2L_g}.
	\end{align}
	On the other hand, by the strong convexity of $f$ and Corollary \ref{th:distance_corollary} we obtain the bounds
	\begin{align} \label{eq:smooth+strong_s+1_minus_star}
		\Delta_s &\le \langle -g_s, \Exp_{x_s}^{-1}(x^*)\rangle - \frac{\mu}{2} d^2(x_t,x^*)\\
			& \le \frac{L_g-\mu}{2}d^2(x_{s},x^*) - \frac{L_g}{2} d^2(x_{s+1},x^*) + \frac{\zeta(\kappa,D)\|g_s\|^2}{2L_g}.
	\end{align}
	Multiply (\ref{eq:smooth+strong_s+1_minus_s}) by $\zeta(\kappa,D)$ and add to (\ref{eq:smooth+strong_s+1_minus_star}) to obtain
	\begin{equation} \label{eq:smooth+strong_telescoping_element}
		\zeta(\kappa,D)\Delta_{s+1} - (\zeta(\kappa,D)-1)\Delta_s \le \frac{L_g-\mu}{2}d^2(x_{s},x^*) - \frac{L_g}{2}d^2(x_{s+1},x^*)
	\end{equation}
	Let $\epsilon = \min\{\frac{1}{\zeta(\kappa,D)}, \frac{\mu}{L_g}\}$, multiply (\ref{eq:smooth+strong_telescoping_element}) by $(1-\epsilon)^{-(s-1)}$ and sum over $s$ from $1$ to $t-1$, we get
	\begin{align}
		\zeta(\kappa,D)(1-\epsilon)^{-(t-2)}\Delta_t \le (\zeta(\kappa,D)-1)\Delta_1 + \frac{L_g-\mu}{2}d^2(x_{1},x^*).
	\end{align}
	Observe that since $\Delta_1\le\frac{L_g D^2}{2}$, it follows that $ \Delta_t \le \frac{(1-\epsilon)^{t-2}L_gD^2}{2}$, as desired.
\end{proof}

It must be emphasized that the proofs of Theorems \ref{th:smoothcvx_rate}, \ref{th:stosmoothcvx_rate}, and \ref{th:smoothstronglycvx_rate} contain some additional difficulties beyond their Euclidean counterparts. In particular, the term $\Delta_{s}$ does not cancel nicely due to the presence of the curvature term $\zeta(\kappa,D)$, which necessitates use of a different Lyapunov function to ensure convergence. Consequently, the stochastic gradient algorithm in Theorem \ref{th:stosmoothcvx_rate} requires some unusual looking averaging scheme. In Theorem \ref{th:smoothstronglycvx_rate}, since the distance between iterates and the minimizer is shrinking, better dependency on $\kappa$ may also be possible if one replaces $\zeta(\kappa,D)$ by a tighter constant.

\section{Experiments}

To empirically validate our results, we compare the performance of a stochastic gradient algorithm with a full gradient descent algorithm on the matrix Karcher mean problem. Averaging PSD matrices have applications in averaging data of anisotropic symmetric positive-definite tensors, such as in diffusion tensor imaging \citep{pennec2006riemannian,fletcher2007riemannian} and elasticity theory \citep{cowin1997averaging}. The computation and properties of various notions of geometric means have been studied by many (e.g. \cite{moakher2005differential,bini2013computing,sra2015conic}). Specifically, the Karcher mean of a set of $N$ symmetric positive definite matrices $\{A_i\}_{i=1}^N$ is defined as the PSD matrix that minimizes the sum of squared distance induced by the Riemannian metric:
	\[ d(X,Y) = \|\log(X^{-1/2}YX^{-1/2})\|_F \]
The loss function
	\[ f(X;\{A_i\}_{i=1}^N) = \sum_{i=1}^N \|\log(X^{-1/2}A_iX^{-1/2})\|_F^2 \]
is known to be nonconvex in Euclidean space but geometrically $2N$-strongly convex
, enabling the use of geometrically convex optimization algorithms. The full gradient update step is 
	\[ X_{s+1} = X_s^{1/2}\exp\left(-\eta_s\sum_{i=1}^N \log(X_s^{1/2}A_i^{-1}X_s^{1/2})\right)X_s^{1/2} \]
For stochastic gradient update, we set
	\[ X_{s+1} = X_s^{1/2}\exp\left(-\eta_sN\log(X_s^{1/2}A_{i(s)}^{-1}X_s^{1/2})\right)X_s^{1/2} \]
where each index $i(s)$ is drawn uniformly at random from $\{1,\dots,N\}$. The step-sizes $\eta_s$ for gradient descent and stochastic gradient method have to be chosen according to the smoothness constant or the strongly-convex constant of the loss function. Unfortunately, unlike the Euclidean square loss, there is no cheap way to compute the smoothness constant exactly. In \citep{bini2013computing} the authors proposed an adaptive procedure to estimate the optimal step-size. Empirically, however, we observe that an $L_g$ estimate of $5N$ always guarantees convergence.
We compare the performance of three algorithms that can be applied to this problem:
\begin{itemize}
	\item Gradient descent (GD) with $\eta_s=\frac{1}{5N}$ set according to the estimate of the smoothness constant (Theorem \ref{th:smoothstronglycvx_rate}).
	\item Stochastic gradient method for smooth functions (SGD-sm) with $\eta_s=\frac{1}{N(s+1)}$ set according to the estimates of the smoothness constant, domain diameter and gradient variance (Theorem \ref{th:stosmoothcvx_rate}).
	\item Stochastic subgradient method for strongly convex functions (SGD-st) with $\eta_s=\frac{1}{N(s+1)}$ set according to the $2N$-strong convexity of the loss function (Theorem \ref{th:stostronglycvx_rate}).
\end{itemize}
 Our data are $100\times 100$ random PSD matrices generated using the Matrix Mean Toolbox \citep{bini2013computing}. All matrices are explicitly normalized so that their norms all equal $1$. We compare the algorithms on four datasets with $N\in\{10^2,10^3\}$ matrices to average and the condition number $Q$ of each matrix being either $10^2$ or $10^8$. For all experiments we initialize $X$ using the arithmetic mean of the dataset. Figure \ref{fig:karcher_mean} shows $f(X) - f(X^*)$ as a function of number of passes through the dataset. We observe that the full gradient algorithm with fixed step-size achieves linear convergence, whereas the stochastic gradient algorithms have a sublinear convergence rate, but is much faster during the initial steps.

\begin{figure}[htp] \label{fig:karcher_mean}
	\vskip 0.2in
	\begin{center}
		\centerline{\includegraphics[width=0.9\columnwidth]{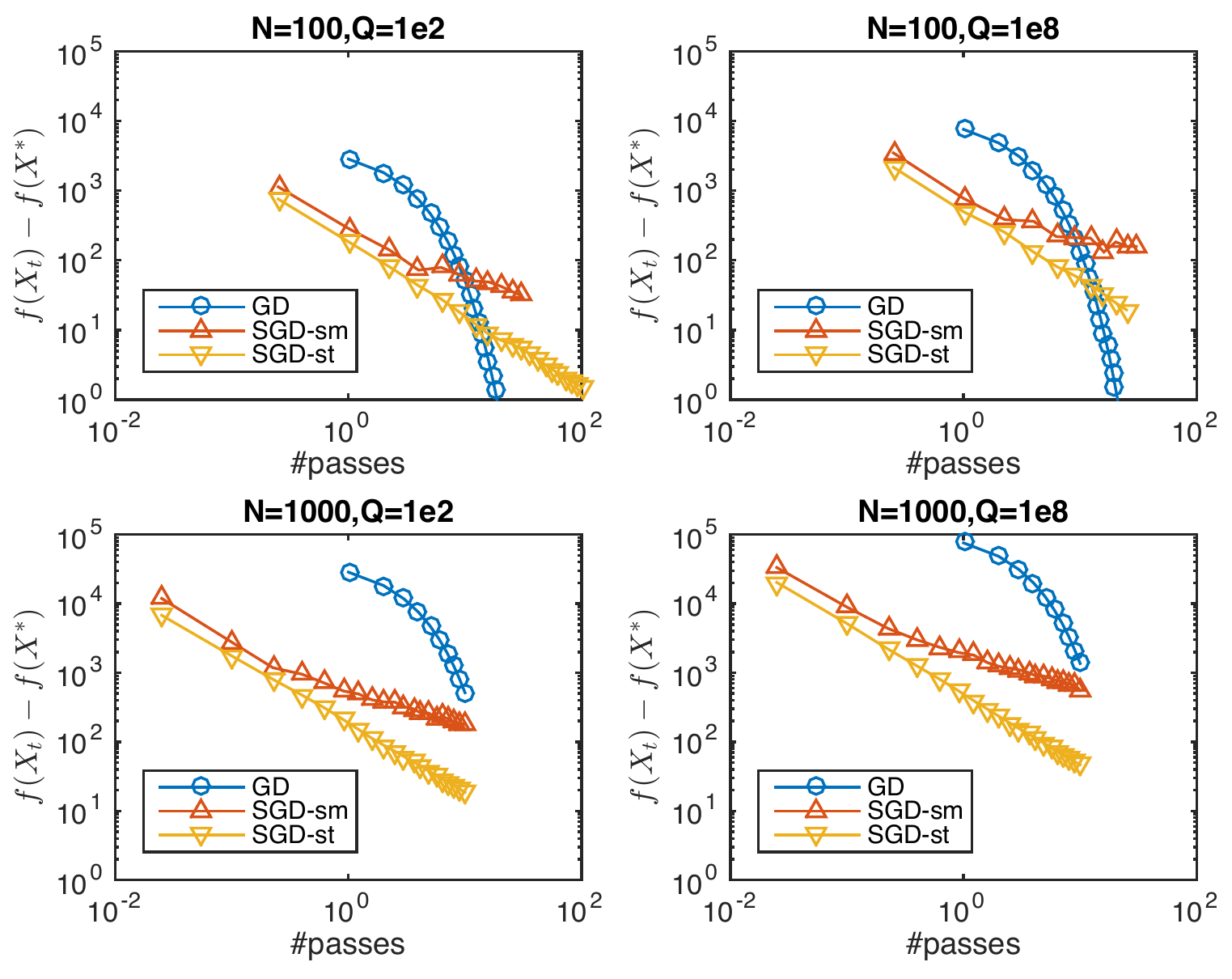}}
		\caption{Comparing gradient descent and stochastic gradient methods in matrix Karcher mean problems. Shown are loglog plots of three algorithms on different datasets. GD: gradient descent (Theorem \ref{th:smoothstronglycvx_rate}); SGD-sm: stochastic gradient method for smooth functions (Theorem \ref{th:stosmoothcvx_rate}); SGD-st: stochastic (sub)gradient method for strongly convex functions (Theorem \ref{th:stostronglycvx_rate}). We varied two parameters: size of the dataset $n\in\{10^2,10^3\}$ and conditional number $Q\in\{10^2,10^8\}$. Data generating process, initialization and step-size are described in the main text. It is validated from the figures that GD converges at a linear rate, SGD-sm converges asymptotically at the $O(1/\sqrt{t})$ rate, and SGD-st converges at the $O(1/t)$ rate.}
		\label{icml-historical}
	\end{center}
	\vskip -0.2in
\end{figure}

\section{Discussion}
In this paper, we make contributions to the understanding of geodesically convex optimization on Hadamard manifolds. Our contributions are twofold: first, we develop a user-friendly trigonometric distance bound for Alexandrov space with curvature bounded below, which includes several commonly known Riemannian manifolds as special cases; second, we prove iteration complexity upper bounds for several first-order algorithms on Hadamard manifolds, which are the first such analyses up to the best of our knowledge. We believe that our analysis is a small step, yet in the right direction, towards understanding and realizing the power of optimization in nonlinear spaces.

\subsection{Future Directions}
Many questions are not yet answered. We summarize some important ones in the following:
\begin{itemize}
	\item A long-time question is whether the famous Nesterov's accelerated gradient descent algorithms have nonlinear space counterparts. The analysis of Nesterov's algorithms typically relies on a proximal gradient projection interpretation. In nonlinear space, we have not been able to find an analogy to such a projection. Further study is needed to see if similar analysis can be developed, or a different approach is required, or Nesterov's algorithms have no nonlinear space counterparts.
	\item Another interesting direction is variance reduced stochastic gradient methods for geodesically convex functions. For smooth and convex optimization in Euclidean space, these methods have recently drawn great interests and enjoyed remarkable empirical success. We hypothesize that similar algorithms can achieve faster convergence over naive incremental gradient methods on Hadamard manifolds.
	\item Finally, since in applications it is often favorable to replace exponential mapping with computationally cheap retractions, it is important to understand the effect of this approximation on convergence rate. Analyzing this effect is of both theoretical and practical interests.
\end{itemize}


\bibliography{gopt,gopt_more}

\appendix
	\section{Proof of Lemma 1}

\begin{lemma} \label{th:lemma_super_exponential}
	Let $$g(b,c) = \cosh \sqrt{\frac{c}{\tanh(c)}b^2 + c^2 - 2bc\cos(A)}$$
	then $$\frac{\partial^2}{\partial b^2}g(b,c) \geq  g(b,c), \quad b,c \geq 0$$
\end{lemma}
\begin{proof}
	If $c=0$, $g(b,c) = \cosh(b) = \frac{\partial^2}{\partial b^2}g(b,c)$. Now we focus on the case when $c>0$. If $c>0$, Let $u = \sqrt{(1+x)b^2 + c^2 - 2bc\cos(A) }$ where $x = x(c)$. We have
	\[u^2 = (1+x)b^2 - 2bc\cos(A) + c^2 \geq \frac{c^2(x+\sin^2A)}{1+x} = u^2_{\min} > 0\]
	
	$$\frac{\partial^2}{\partial b^2}g(b,c) = \left(1+x-c^2\left(x+\sin^2A\right)\frac{1}{u^2} \right) \cosh(u) + c^2\left(x+\sin^2A\right) \frac{\sinh(u)}{u^3}$$
	Since $g(b,c)=\cosh(u)>0$, it suffices to prove
	\[\frac{\frac{\partial^2}{\partial b^2}g(b,c)}{g(b,c)} - 1 = x\left(1 - \frac{c^2}{x}\left(x+\sin^2A\right)\frac{1}{u^2} + \frac{c^2}{x}\left(x+\sin^2A\right) \frac{\tanh(u)}{u^3}\right) \geq 0\]
	so it suffices to prove
	\[h_1(u) = \frac{u^3}{u-\tanh(u)}\geq \frac{c^2}{x}(x+\sin^2A)\]
	Solving for $h'_1(u) = 0$, we get $u=0$. Since $\lim_{u\to 0_+} h_1(u) = 0$ and $h_1(u)>0, \forall u>0$, $h_1(u)$ is monotonically increasing on $u>0$. Thus $h_1(u) \geq h_1(u_{\min}), \forall u>0$.
	Note that $\frac{c^2}{x}(x+\sin^2A) = \frac{1+x}{x}u^2_{\min}$, thus it suffices to prove
	\[h_1(u_{\min}) = \frac{u^3_{\min}}{u_{\min}-\tanh(u_{\min})} \geq \frac{(1+x)u^2_{\min}}{x}\]
	or equivalently
	\[\frac{\tanh(u_{\min})}{u_{\min}} \geq \frac{1}{1+x}\]
	Now fix $c$ and notice that $\frac{\tanh(u_{\min})}{u_{\min}}$ as a function of $\sin^2A$ is monotonically decreasing. Therefore its minimum is obtained at $\sin^2A=1$, where $u^2_{\min} = u^2_* = c^2$, i.e. $u_* = c$. So it only remains to show
	\[\frac{\tanh(u_*)}{u_*} = \frac{\tanh(c)}{c} \geq \frac{1}{1+x}, \forall c>0\]
	or equivalently
	\[  1+x \geq \frac{c}{\tanh(c)}, \forall c>0 \]
	which is true by our definition of $g$.
\end{proof}
\begin{lemma} \label{th:lemma_differential_inequality}
	Suppose $h(x)$ is twice differentiable on $[r, +\infty)$ with three further assumptions:
	\begin{enumerate}
		\item $h(r) \leq 0$,
		\item $h'(r) \leq 0$,
		\item $h''(x) \leq h(x), \forall x \in [r, +\infty)$,
	\end{enumerate}
	then $h(x) \leq 0, \forall x \in [r, +\infty)$
\end{lemma}
\begin{proof}
	It suffices to prove $h'(x) \leq 0, \forall x \in [r, +\infty)$. We prove this claim by contradiction.
	
	Suppose the claim doesn't hold, then there exist some $t>s\geq r$ so that $h'(x) \leq 0$ for any $x$ in $[r,s]$, $h'(s) = 0$ and $h'(x) > 0$ is monotonically increasing in $(s,t]$. It follows that for any $x \in [s, t]$ we have
	\[h''(x) \leq h(x) \leq \int_r^x h'(u) du \leq \int_s^x h'(u) du \leq (x-s) h'(x) \leq (t-s) h'(x)\]
	Thus by Gr\"onwall's inequality, $$h'(t) \leq h'(s)e^{(t-s)^2} = 0$$
	which leads to a contradiction with our assumption $h'(t) > 0$.
\end{proof}

\begin{lemma} \label{th:canonic_hyperbolic_distance_bound}
	If $a,b,c$ are the sides of a (geodesic) triangle in a hyperbolic space of constant curvature $-1$, and $A$ is the angle between $b$ and $c$, then
	\[a^2 \leq \frac{c}{\tanh(c)} b^2 + c^2 - 2bc\cos(A)\]
\end{lemma}
\begin{proof}
	For a fixed but arbitrary $c\geq 0$, define $h_c(x) = f(x,c) - g(x,c)$. By Lemma \ref{th:lemma_super_exponential} it is easy to verify that $h_c(x)$ satisfies the assumptions of Lemma \ref{th:lemma_differential_inequality}. Apply Lemma \ref{th:lemma_differential_inequality} to $h_c$ with $r=0$ to show $h_c \leq 0$ in $[0, +\infty)$. Therefore $f(b,c)\leq g(b,c)$ for any $b, c \geq 0$. Finally use the fact that $\cosh(x)$ is monotonically increasing on $[0, +\infty)$.
\end{proof}

\begin{corollary} \label{th:hyperbolic_distance_bound}
	If $a,b,c$ are the sides of a (geodesic) triangle in a hyperbolic space of constant curvature $\kappa$, and $A$ is the angle between $b$ and $c$, then
	\[a^2 \leq \frac{\sqrt{|\kappa|}c}{\tanh(\sqrt{|\kappa|}c)} b^2 + c^2 - 2bc\cos(A)\]
\end{corollary}
\begin{proof}
	For hyperbolic space of constant curvature $\kappa < 0$, the law of cosines is
	\[ \cosh (\sqrt{|\kappa|} a) = \cosh (\sqrt{|\kappa|} b) \cosh (\sqrt{|\kappa|} c) - \sinh (\sqrt{|\kappa|} b) \sinh (\sqrt{|\kappa|} c) \cos A \]
	which corresponds to the law of cosines of a geodesic triangle in hyperbolic space of curvature $-1$ with side lengths $\sqrt{|\kappa|} a, \sqrt{|\kappa|} b, \sqrt{|\kappa|} c$. Applying Lemma \ref{th:canonic_hyperbolic_distance_bound} we thus get
	\[ |\kappa|a^2 \leq \frac{\sqrt{|\kappa|}c}{\tanh(\sqrt{|\kappa|}c)} |\kappa|b^2 + |\kappa|c^2 - 2|\kappa|bc\cos(A) \]
	and the corollary follows directly.
\end{proof}

\end{document}